\documentclass[11pt]{amsart}
\usepackage{amsmath,amssymb, graphicx, amscd,latexsym,comment}
\makeatletter
\newtheorem{Theorem}{Theorem}
\newtheorem{Lemma}[Theorem]{Lemma}

\newtheorem{Proposition}[Theorem]{Proposition}

\newtheorem{Definition}[Theorem]{Definition}
\newtheorem{Remark}[Theorem]{Remark}

\newtheorem{Example}[Theorem]{Example}

\newtheorem{Assertion}[Theorem]{Assertion}

\newcommand{\eps}{\varepsilon}

\newcommand\la{\lambda}
\newcommand\vphi{\varphi}

\newcommand\al{\alpha}

\newcommand\be{\beta}
\newcommand\Si{\Sigma}

\newcommand\Ga{\Gamma}
\newcommand\de{\delta}
\newcommand\De{\Delta}

\newcommand\cS{\mathcal  S}

\newcommand\BC{ {\mathbb C}}

\newcommand\bfp{\mbox {\bf  p}}

\newcommand\bfq{\mbox {\bf  q}}

\newcommand\bfv{\mbox {\bf  v}}

\newcommand\bfx{\mbox {\bf  x}}

\newcommand\bfw{\mbox {\bf  w}}

\newcommand\bfz{\mbox {\bf  z}}
\newcommand\bfy{\mbox {\bf  y}}
\newcommand\bfa{\mbox {\bf  a}}

\newcommand\bfb{\mbox {\bf  b}}

\newcommand\nl{\newline}

\newcommand\ord{{\rm{ord}\/}}
\newcommand\Int{\rm{Int}\/}

\newcommand\grad{{\rm{grad}\/}}

\newcommand\Cone{\rm{Cone}\/}

\newcommand\pdeg{{\rm{pdeg}\/}}

\newcommand\inv{^{-1}}

\def\mapright#1{\smash{\mathop{\longrightarrow}\limits^{{#1}}}}

\def\inv{^{-1}}

\begin{document}
\title[On Milnor fibrations, $a_f$-condition and boundary stability
]
{On Milnor fibrations of  mixed functions, $a_f$-condition and boundary stability
}

\author
[M. Oka ]
{Mutsuo Oka }
\address{\vtop{
\hbox{Department of Mathematics}
\hbox{Tokyo  University of Science}
\hbox{1-3, Kagurazaka, Shinjuku-ku}
\hbox{Tokyo 162-8601}g
\hbox{\rm{E-mail}: {\rm oka@rs.kagu.tus.ac.jp}}
}}


\begin{abstract}
Convenient mixed functions with strongly non-degenerate Newton boundaries have  a Milnor fibration (\cite{OkaMix}), 
as the isolatedness of the singularity follows from the convenience. In this paper,  we consider the Milnor fibration for non-convenient
mixed functions
We also study geometric properties such as Thom's $a_f$ condition, the transversality of the nearby fibers and stable boundary property of 
the Milnor fibration and their relations.
\end{abstract}
\maketitle

\maketitle
  \noindent
\section{Preliminary}
Let $f(\bfz,\bar\bfz)$ be a mixed function and write it as sum of real and imaginary part: $f=g+ih$.  Writing $\bfz=(z_1,\dots, z_n)$ and  $z_j=x_j+iy_j\,(j=1,\dots,n)$
with $x_j,y_j\in \mathbb R$,
the mixed hypersurface $\{f=0\}$ can be understand as the real analytic variety in $\mathbb R^{2n}$ defined by  $\{g=h=0\}$. The real and imaginary part $g,h$
are also (real-valued) mixed functions and  we also consider them as real analytic functions of variables 
$\bfx=(x_1,\dots, x_n)$ and $\bfy=(y_1,\dots, y_n)$. By abuse of notations we use both notations
$g(\bfz,\bar \bfz)$ and $g(\bfx,\bfy)$ etc. We recall some notations. The {\em real gradient vector} for a real-valued mixed function 
$k(\bfx,\bfy)$ is defined as 
\begin{eqnarray}
\grad\, k&=&(\grad_{\bfx}k,\grad_{\bfy}k)\in \mathbb R^{2n}\\
 \grad_{\bfx}k&=&(k_{x_1},\dots, k_{x_n}), \,\,\grad_{\bfy}k=(k_{y_1},\dots,k_{y_n}).
\end{eqnarray}
Here $k_{x_i}, k_{y_j}$ are respective partial derivatives.
$\mathbb C^n$ and $\mathbb R^{2n}$ are identified by $\bfz\leftrightarrow \bfz_{\mathbb R}=(\bfx,\bfy)$.
Under this identification, the Euclidean inner product in $\mathbb R^{2n}$ (denoted as $(*,*)_{\mathbb R}$)
and the hermitian inner product in $\mathbb C^n$ (denoted as $(*,*)$) are related as
$(\bfz_{\mathbb R},\bfz_{\mathbb R}')_{\mathbb R}=\Re (\bfz,\bfz')$.
For  a mixed function $k$ (not necessarily real-valued), we define also 
holomorphic and anti-holomorphic gradients as
\[\begin{split}
&\grad_\partial k=(\frac{\partial k}{\partial z_1},\dots, \frac{\partial k}{\partial z_n}),\\
&\grad_{\bar\partial} k=(\frac{\partial k}{\partial \bar z_1},\dots, \frac{\partial k}{\partial\bar z_n}).
\end{split}
\]

For simplicity of notations, we use the following notations:
\[\begin{split}
&dk:=\grad\, k,\, \, d_{\bfx}k:=\grad_{\bfx} k,\,\, d_{\bfy}k:=\grad_{\bfy}k,\,\,\\
&\partial k:=\grad_{\partial}k,\,\, {\bar\partial} \,k:=\grad_{\bar\partial} k.
\end{split}
\]
Note that if $k$ is real-valued,
\begin{eqnarray}\label{real-function-gradient}
\overline{\partial k}={\bar\partial} k,\end{eqnarray}
and 
real vector $dk\in\mathbb R^{2n}$ corresponds to the complex vector $ 2 \overline{\partial k}\in \mathbb C^n$.
\subsubsection{Tangent spaces}
Let $k(\bfz,\bar\bfz)$ is a real valued mixed function.
Then the tangent space of a regular point $\bfa\in V_\eta:=k\inv(\eta),\eta\in \mathbb R$ is described as follows.
For a complex vector  $\bfa\in\mathbb C^{n}$, we denote the corresponding real vector as 
$\bfa_{\mathbb R}\in \mathbb R^{2n}$.
\[
\begin{split}
T_{\bfa} V_k
&=\{\bfv_{\mathbb R}\in \mathbb R^{2n}\,|\ (\bfv_{\mathbb R}, dk(\bfa_{\mathbb R}))_{\mathbb R}=0\}\\
&=\{\bfv\in \mathbb C^{n}\,|\,\Re(\bfv,\overline{\partial k}(\bfa))=0\}.
\end{split}
\]
Consider the mixed hypersurface $V_\eta=f\inv(\eta),\,\eta\ne 0$.
We introduce two vectors in $\mathbb C^n$ which are more convenient to describe the Milnor fibration of the first type:
\[\begin{split}
&\bfv_1:=\overline{\partial f}(\bfz,\bar\bfz)+\bar\partial f(\bfz,\bar\bfz),\\
&\bfv_2:=i(\overline{\partial f}(\bfz,\bar\bfz)-\bar \partial f(\bfz,\bar\bfz)).
\end{split}
\]
These vectors describe the respective tangent spaces at a regular point $\bfa$ of the real codimension 1 varieties
\[
\begin{split}
&V_1:=\{\bfz\,|\, |f(\bfz,\bar\bfz)|=|f(\bfa,\bar\bfa)|\},\\
&V_2:=\{\bfz\,|\, \arg\,f(\bfz,\bar\bfz)=\arg\eta\}.
\end{split}
\]
Namely,
we have shown (Lemma 30,Observation 32,\cite{OkaMix})
\[
\begin{split}
&T_{\bfa}V_1:=\{\bfv\,|\, \Re(\bfv,\bfv_1(\bfa))=0\}\\
&T_{\bfa}V_2:=\{\bfv\,|\, \Re(\bfv,\bfv_2(\bfa))=0\}.
\end{split}
\]
Note that $V_{\eta}=V_1\cap V_2$.
Observe that
the two subspaces  of dimension two
\[<\bar\partial g(\bfa,\bar\bfa),\bar \partial h(\bfa,\bar\bfa)>_{\mathbb R},\,<\bfv_1(\bfa),\bfv_2(\bfa)>_{\mathbb R}
\]
are equal. In fact we have:
\begin{eqnarray*}
\bfv_1&=&\frac{\overline{ \partial f}} {\bar f} + \frac { \bar \partial f}f=\frac 1{|f|^2} (f(\overline{\partial g}-i \overline{\partial h})+\bar f(\bar \partial g+i \bar\partial h))=\frac 1{|f|^2}(2g\bar \partial g+2h \bar \partial h)\\
\bfv_2&=&i\frac {\overline{ \partial f}}{\bar f}+ \frac {\bar \partial f}f=\frac i{|f|^2} (f(\bar\partial g-i \bar\partial h)+\bar f(\bar \partial g+i \bar\partial h))=\frac 1{|f|^2}(-2h\bar \partial g-2g \bar \partial h)
\end{eqnarray*}
\begin{Proposition}(\cite{OkaPolar})\label{mixed critical}
Put $f=g+hi$ as before.
The next conditions are equivalent.
\begin{enumerate}
\item $\bfa\in \mathbb C^n$ is a critical point of the mapping $f:\mathbb C^n\to \mathbb C$.
\item $dg(\bfa_{\mathbb R}),\,dh(\bfa_{\mathbb R})$ are linearly dependent over $\mathbb R$.
\item $\bar\partial g(\bfa,\bar \bfa),\,\bar\partial h(\bfa,\bar \bfa)$ are linearly dependent over $\mathbb R$.
\item There exists a complex number $\al$ with $|\al|=1$ such that 
$\overline{\partial f}(\bfa,\bar \bfa)=\al\,\bar \partial f(\bfa,\bar \bfa)$.
\end{enumerate}
\end{Proposition}
Under the above equivalent conditions, we say that $\bfa$ is a {\em mixed singular point} of the mixed hypersurface
$f\inv(f(\bfa))$.
\begin{Lemma} (cf \cite{Chen})
Put $V_\eta=f\inv(\eta)$ and take $\bfp\in S_r\cap V_\eta$.  Assume that $\bfp$ is a non-singular point of $V_\eta$
and let $k(\bfz,\bar\bfz)$ be a real valued mixed function.
The following conditions are equivalent.
\begin{enumerate} 
\item The restriction $k|V_\eta$ has a critical point at  $\bfa\in V_\eta$.
\item There exists a complex number $\al\in \mathbb C^*$ such that 
$\bar\partial k(\bfp)=\al\overline{ \partial f}(\bfp,\bar\bfp)+\bar \al\bar\partial f(\bfp,\bar\bfp)$.
\item There exist real numbers $c, d$ such that 
\[\bar\partial k(\bfp)=c\bar\partial  g(\bfp,\bar\bfp)+d\bar\partial h(\bfp,\bar\bfp).\]
\item There exist real numbers $c', d'$ such that 
\[\bar\partial k(\bfp)=c'\bar\bfv_1(\bfp,\bar\bfp)+d'\bfv_2(\bfp,\bar\bfp).\]

\end{enumerate}
\end{Lemma}
\begin{proof}
As $\bfp\in V$ is assumed a non-singular point,
(1) and (3) are equivalent. We show the implication (3) $\implies$ (2). Assume
\[
\bar\partial k(\bfp)=c\,\bar \partial g(\bfp,\bar\bfp)+d\,\bar \partial h(\bfp,\bar\bfp),\quad \exists c,d\in \mathbb R.\]
We use the equality:
\begin{eqnarray}\label{eq3}
&\bar\partial g(\bfp,\bar\bfp)&=\frac{\bar\partial f(\bfp,\bar\bfp)+\overline{\partial f}(\bfp,\bar\bfp)}2,\\
&\bar\partial h(\bfp,\bar\bfp)&=\frac{i(\bar\partial f(\bfp,\bar\bfp)-\overline{\partial f}(\bfp,\bar\bfp))}2
\end{eqnarray}
to obtain the equality:
\[\bar\partial k(\bfp)=\frac{c-di}2 \bar\partial f(\bfp,\bar\bfp)+\frac{c+di}2 \overline{\partial f}(\bfp,\bar\bfp).
\]
The implication $(2)\to (3)$ can be shown similarly, using the equality
\begin{eqnarray}\label{eq5}
\partial f=\partial g+i \partial h,\bar \partial f=\bar\partial g+i \bar \partial h
\end{eqnarray}
\end{proof}
\subsubsection{Newton boundary and strong non-degeneracy condition.}
Let $f(\bfz,\bar\bfz)=\sum_{\nu,\mu} c_{\nu \mu} \bfz^\nu{\bar\bfz}^\mu$ be a mixed polynomial.
The Newton polygon $\Ga_+(f)$ is defined by the convex hull of
$\bigcup(\nu+\mu+\mathbb R_+^n)$ where the sum is taken for $\nu,\mu$ with $c_{\nu \mu}\ne 0$.
Newton boundary $\Ga(f)$ is the union of compact faces of $\Ga_+(f)$ as usual.
$f$ is called {\em convenient} if for any $i=1,\dots, n$, $\Ga(f)$ intersects with  $z_i$-axis.
To treat the case of non-convenient functions,
we define {\em the modified Newton boundary}  $\Ga_{nc}(f)$ by adding essential non-compact faces $\Xi$. Here
$\Xi$ is called {\em an essential non-compact face} if  there exists a semi-positive weight vector $P={}^t(p_1,\dots, p_n)$
such that 
\begin{enumerate}
\item $\De(P)=\Xi$ with $\Xi$ being a non-compact face  
and $f^{I(P)}\equiv 0$
where $I(P)=\{i\,|\, p_i=0\}$ 
and  
\item for any $i\in I(P)$ and any point $\nu\in \Xi$,   the half line starting from $\nu$,
$\nu+\mathbb R_+ E_i$ is contained in $\Xi$. Here $E_i$ is the unit vector in the direction of $i$-th  coordinate axis.
\end{enumerate}
The weight vector $P$ may not unique but $I(P)$ does not depend on $P$. Thus we denote it as $I(\Xi)$
and it is called {\em the non compact direction } of $\Xi$. See Figure 1 which shows the modified Newton boundary of 
$f=z_1^3+z_2^3+z_2z_3^2$ in Example 3.

For any non-negative weight vector $P$, it defines a linear function $\ell_P$ on $\Ga_{+}(f)$
by $\ell_P(\xi)=p_1\xi_1+\cdots+p_n\xi_n$ 
where $P={}^t(p_1,\dots, p_n),\,\xi=(\xi_1,\dots,\xi_n)\in \Ga(f)$ and the minimal value is denoted as $d(P)$ and 
the face where this minimal value is taken is denoted by $\De(P)$.
In other word, $\De(P):=\{\xi\in \Ga_{nc}(f)\,|\, \sum_{i=1}^n p_i\xi_i=d(P)\}$.
The face function associated by $P$ is defined as $f_P:=f_{\De(P)}$.

$f$ is called {\em strongly non-degenerate} if (1) for any  compact face $\De\subset \Ga_{nc}(f)$,
the face function $f_\De:=\sum_{\nu+\mu\in \De}c_{\nu \mu}\bfz^\nu{\bar \bfz}^\mu$ has no
critical point as a function $f_\De:\mathbb C^{*n}\to \mathbb C$ and 
(2) for a non-comact face $\De\in \Ga_{nc}(f)$,  $f_{\De_0}:\mathbb C^{*n}\to \mathbb C$ has no critical point
where $\De_0=\De\cap \Ga(f)$. 
\begin{Example}
Consider a holomorphic function $f=z_1^3+z_2^3+z_2z_3^2$ of three variables. Note that 
$\Ga_{nc}(f)$ has three vertices $A=(3,0,0), B=(0,3,0), C=(0,1,2)$ and the face
$\De:=\{\overline {AC}+\mathbb R_+\,E_3\}\subset\Ga_{nc}(f)$ where 
$\overline {AC}$ is the edge with endpoints $A,C$. 
The non-compact faces with edge $\overline{AB}$ and $\overline{BC}$ are not essential.
They are not vanishing coordinates i.e.,   $f$ does not vanish on $\{z_1=z_3=0\}$ or $\{z_2=z_3=0\}$.
See Figure 1.
\end{Example}

\begin{figure}
{\includegraphics[width=6cm,height=5cm]{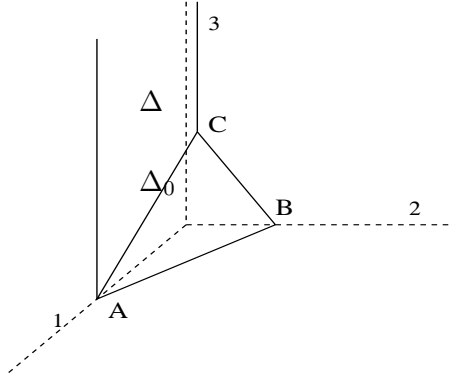}}
\put (-120,70) {$\De_0$}
\put (-120,100) {$\De$}
\caption{\label{Non-compact face}$Non-compact face$}
\end{figure}
\section{Milnor fibration}
Asume  that  $f(\bfz,\bar\bfz)=\sum_{\nu,\mu} c_{\nu \mu}\bfz^\nu{\bar\bfz}^\mu$ is
 a strongly
 non-degenerate mixed polynomial and let $V=f\inv(0)$.
In this section, we study the Milnor fibration of $f$.
 If $f(\bfz,\bar\bfz)$ has a convenient Newton boundary,
the singularity is isolated and there exists a spherical Milnor fibration  (= a  Milnor fibrations
of the first type): 
\[
f/|f|: S_r-K\to S^1,\quad K=V\cap S_r
\]
and also a tubular Milnor fibration (= Milnor fibration of the second type):
$f: \partial E(r,\de)^*\to S_\de^1$ where
$\partial E(r,\de)^*=\{\bfz\in B_r\,|\, |f(\bfz,\bar\bfz)|=\de\}$ for  sufficiently small $r,\de$ such that 
$0<\de\ll r$. They are $C^\infty$-equivalent (Theorems 19, 33,  37,\cite{OkaMix}).

For non-convenient mixed function, the singularity need not be isolated.
We have proved the same assertion under an extra condition ``super strongly non-degenerate'' (Theorem 52,\cite{OkaMix}).
In this paper, we prove the existence of Milnor fibrations  for any strongly non-degenerate functions
with a weaker assumption than
the assumption ``super''. We will study also some geometric properties behind the argument.
\subsection{Smoothness of the nearby fibers}

First we recall the following:
\begin{Lemma}(Lemma 28,\cite{OkaMix})\label{MilnorFibering2}
Assume that  $f(\bfz,\bar\bfz)$ is
 a strongly
 non-degenerate 
mixed function.
Then there exists a positive number $r_0$ and $\de$ such that
the fiber $V_\eta:=f\inv(\eta)$ has no mixed singularity in the ball
 $B_{r_0}^{2n}$ for any non-zero $\eta$ with $|\eta|\le \de$. 
\end{Lemma}

\begin{proof} Though the proof is the same as that in \cite{OkaMix},  we repeat it for the beginner's convenience.
 We show  a contradiction ,  assuming that 
 the assertion does not hold. Then using  the Curve Selection
 Lemma (\cite{Milnor,Hamm1}),  we can find an analytic  path
$\bfz(t),\,0\le t\le 1$ such that 
$\bfz(0)=O$ and 
$f(\bfz(t),\bar\bfz(t))\ne 0$  and $\bfz(t)$ is a critical point
 of the function $f:\BC^n\to \BC$ for any $t\ne 0$.
Using Proposition \ref{mixed critical},
we can find a real analytic  family $\la(t)$ in $S^1\subset \BC$
such that 
\begin{eqnarray}\label{Sing-cond1}
 \overline{\partial f}(\bfz(t),\bar\bfz(t))=\la(t)\, \bar \partial f(\bfz(t),\bar\bfz(t)).
\end{eqnarray}
Put $I=\{j\,|\,z_j(t)\not \equiv 0\}$.
We may assume for simplicity that 
 $I=\{1,\dots,m\}$ and we consider the restriction  $f^I=f|\mathbb C^I$. 
As $f(\bfz(t),\bar \bfz(t))=f^I(\bfz(t),\bar \bfz(t))\not \equiv 0$, we
 see that $f^I\ne 0$. 
Consider the Taylor expansions of $\bfz(t)$ and $\la(t)$:
\begin{eqnarray*}
 &\bfz_i(t)=b_i \,t^{a_i}+\text{(higher terms)},\,b_i\ne 0,a_i>0,\quad
i=1,\dots, m\\
&\la(t)=\la_0+\text{(higher terms)},\quad \la_0\in S^1\subset \BC.
\end{eqnarray*}
Consider the weight vector  $A={}^t(a_1,\dots, a_m)$ and a point in the torus $\bfb=(b_1,\dots, b_m)\in \mathbb C^{*I}$
  and  we consider the face function $f^I_A$
of $f^I(\bfz,\bar\bfz)$.
Then 
we have for $j\in I$
\begin{eqnarray*}
&\frac{\partial f}{\partial z_j}(\bfz(t),\bar\bfz(t))=\frac{\partial f^I_A}{\partial
 z_j}(\bfb,\bar\bfb)\,t^{d-a_j}+\text{(higher terms)},\\
 &\frac{\partial f}{\partial \bar z_j}(\bfz(t),\bar\bfz(t))=\frac{\partial f^I_A}{\partial
 z_j}(\bar\bfb,\bar\bfb)\,t^{d-a_j}+\text{(higher terms)}
 \end{eqnarray*}
where  $d=d(A;f^I)$.
The equality (\ref{Sing-cond1}) says that 
\[ {\overline{\frac{\partial f^I}{\partial z_j}}
 \,(\bfz(t),\bar\bfz(t))}=\la(t)\,
 \frac{\partial f^I}{\partial \bar z_j}(\bfz(t),\bar\bfz(t)),\,\,j=1,\dots, m.
\]
which implies the next equality:
\[
 \ord_t\,{\overline{\frac{\partial f^I}{\partial z_j}}
 \,(\bfz(t),\bar\bfz(t))}=\ord_t \,
 \frac{\partial f^I}{\partial \bar z_j}(\bfz(t),\bar\bfz(t)),\,\,j=1,\dots, m.
\]
 Thus we get the equality:
  \[
\overline{\partial f^I_A}(\bfb,\bar \bfb)=\la_0\, \bar \partial f^I_A(\bfb,\bar\bfb),\quad
\bfb\in \BC^{*m}.
\]
  This implies that $\bfb$ is 
 a critical point of $f_A^I:\BC^{*I}\to \BC$, which is a contradiction to the
 strong non-degeneracy of 
$f_A^I(\bfz,\bar\bfz)$.
\end{proof}

\subsection{ Vanishing coordinate subspaces and essentially non-compact face functions}
 We assume that $f$ is a mixed polynomial (not only mixed analytic function).
 We denote by $\mathcal I_{nv}(f)$  the set of subset $I\subset\{1,2,\dots, n\}$ such that $f^I\not \equiv 0$
 (we denoted this set as $\mathcal {NV}(f)$ in \cite{OkaMix}).
 We denote by $\mathcal {I}_v(f)$ the set of subset $I\subset\{1,2,\dots, n\}$ such that $f^I \equiv 0$, and 
  for $I\in \mathcal {I}_v(f)$ and we consider also the set of non-compact faces $\De\in \Ga_{nc}(f)$
 such that there exists (possibly not unique) a non-negative weight $P$ such that
 $\De(P)=\De$ and $I(P)=I$.　Here $I(P)=\{i\,|\, p_i=0\}$.
 $\mathbb C^{I}$ is called {\em a vanishing coordinates subspace}. Note that $\mathbb C^I\subset V$.

\begin{Definition}
Take an essential non-compact face $\De\in \Ga_{nc}(f)$. Take a weight function $P$ such that $f_P=f_{\De}$ and $I(P)=I(\De)$.
We consider the function $\rho_\De(\bfz):=\sum_{j\in I(\De)} |z_j|^2$.
An essential non-compact face function $f_\De$  is {\em locally tame} if there exists a positive number $r_\De>0$ such that for any 
fixed $\{z_j\,|\,z_j\ne 0,\,j\in I(\De)\}$
with $\rho_\De(\bfz)\le r_\De^2$, $f_\De$  has no critical points in $\mathbb C^{*I(\De)^c}$
as a  mixed polynomial function  of 
$n-|I(\De)|$-variables $\{z_k\,|\,k\notin I(\De)\}$,
and we can also assume that the  function $\rho_\De$ has no critical value on $V_\De^*$ on the interval $(0,r_\De^2]$ where $V_\De=f_\De\inv(0)\subset \mathbb C^{*n}$.
We say  that {\em $f$  is locally tame on the vanishing coordinate subspace $\mathbb C^{I}$} if 
any face function $f_\De$ with $I(\De)=I$ is locally tame.
 This is slightly weaker condition than "super strongly non-degenerate" in \cite{OkaMix}.
\end{Definition}
 Put $r_I=\min\,\{r_\De\,|\, I(\De)=I\}$ for $I\in \mathcal I_v(f)$ and 
 $r_{nc}=\min\,\{r_I\,|\, I\in \mathcal I_v(f)\}$.
If $f$ is convenient,  $r_{nc}=+\infty$.
\begin{Remark} We say that
$f$ is  "super strongly non-degenrate" if  we can take $r_\De=\infty$ in the above definition (\cite{OkaMix}).
For the existence of $r_\De$, we used the fact that $f$ is a polynomial.
\end{Remark}
\subsection{Smoothness on the non-vanishing coordinate subspaces}
Take $I\subset \{1,\dots,n\}$ and $\mathbb C^I$ is called a non-vanishing coordinate if $f^I\not \equiv 0$.
The set of such subsets $I$ is denoted as $\mathcal {I}_{nv}(f)$.
Put $V^{\sharp}=\cup_{I\in \mathcal{I}_{nv}(f)}V\cap\mathbb C^{*I}$.
Then there exists a $r_0>0$  so that 
$V^{\sharp}$ and $V^{*I}=V\cap \mathbb C^{*I}$ are  non-singular in the ball $B_{r_0}$ and for any $0<r\le r_0$, the sphere $S_r$  and $V^{*I}$ intersect transversely.
The existence of such $r_0$ is shown  in Theorem 16, \cite{OkaMix}.

\subsection{Hamm-L\^e type  theorem}
The following is a mixed function  version of Lemma (2.1.4) (Hamm-L\^e,\cite{Hamm-Le1}).
This enable us to prove the existence of Milnor fibration with locally tame behavior assumption.
\begin{Lemma}\label{mixed Hamm-Le1}
Assume that $f(\bfz,\bar\bfz)$ is a strongly non-degenerate mixed polynomial which behave locally tamely along vanishing coordinate subspaces.
Put $\rho_0=\min\,\{r_{nc},r_0\}$
where $r_{nc}$ and $r_0$ are described above.
For any fixed positive number $ r_1\le \rho_0$, there
 exists
 positive numbers
 $\de(r_1)$ (depending on $r_1$)
such that
\begin{enumerate}
\item 
the nearby fiber $V_\eta:=f\inv(\eta)$ has no mixed singularity in the ball
 $B_{\rho_0}^{2n}$ for any non-zero $\eta$.
\item 
 for any $\eta\ne 0,\,|\eta|\le \de(r_1)$ and $r,\,r_1\le  r\le \rho_0$, the sphere $S_r$ 
 and the nearby fiber $V_{\eta}=f\inv(\eta)$
 intersect
 transversely.
\end{enumerate}
\end{Lemma}
\begin{proof} We have already proved the assertion (1) (Lemma \ref{MilnorFibering2}).  So we will prove the assertion (2).
Assume that the assertion is false. By the Curve Selection Lemma, we can find a real analytic curve $\bfz(t)$ and a complex valued function
$\al(t),\,0\le t\le 1$ 
\begin{eqnarray}\label{non-transverse}
z_j(t)=\al(t)\frac{\overline{\partial f}}{\partial z_j}(\bfz(t))+\bar\al(t)\frac{\partial f}{\partial\bar z_j}(\bfz(t)),\,\forall j
\end{eqnarray}
 where $\bfz(t), \al(t)$ are expanded as
\[\begin{split}
&z_j(t)=b_jt^{p_j}+\text{(higher terms)},\\
&\al(t)=\al_0t^m+\text{(higher terms)}.
\end{split}
\]
and $f(\bfz(t))\neq 0$ for $t\ne 0$. Obviously $\al(t)\ne 0$.

Put  $K=\{i\,|\, z_j(t)\not \equiv 0\}$and we consider the equality  in $\mathbb C^K$. Put $\bfb=(b_j)$ and $P=(p_j)$, $I=\{j\in K\,|\, p_j=0\}, \,I_1=K-I$
and $\De=\De(P)$. In the following, we assume $K=\{1,\dots,n\}$ as the argument is the same.

Case 1. Assume that $I\in \mathcal{I}_{nv}(f)$. Then $f^I\not \equiv 0$ and $\bfb\in V^\sharp$.
We assumed that $V^{\sharp}$ and $S_{\|\bfb\|}$ intersect transversely
 for
any $\bfb$, $\|\bfb\|\le \rho_0$
and thus $S_{\|\bfz(t)\|}$ is also transverse to $V_{f(\bfz(t))}$ at $\bfz(t)$ for a small $t\ll 1$ which is a contradiction.

Case 2. Assume that $I\in\mathcal{I}_{v}(f)$ and so $f^I\equiv 0$. In this case, $\De\in \Ga_{nc}(f)$. 
The above equality (\ref{non-transverse}) says:
\begin{multline}
b_jt^{p_j}+\text{(higher terms)}=\left (\al_0 \frac{\overline{\partial f_{\De(P)}}}{\partial  z_j}(\bfb) t^{m+d(P)-p_j}+\text{(higher terms)}\right)\\
+ \left (\bar\al_0 \frac{\partial f_{\De(P)}}{\partial \bar z_j}(\bfb) t^{m+d(P)-p_j}+\text{(higher terms)}\right ), \quad j\in K.
\end{multline}
We compare the order  in  $t$  (=the lowest degree) of the both side.
The left side has order $0$ and the order of the right side is 
at least $d+m-p_j$ for $j\notin I$ and at least $d(P)+m$ for $j\in I$.
Note that $\bfb\in \mathbb C^{*n}$.
If $d(P)+m>$, we get a contradiction $b_j=0$ for $j\in I$.
If $d(P)+m<0$,  we get 
\[0=\al_0\frac{\overline{\partial f_{\De(P)}}}{\partial z_j}(\bfb)+\bar
 \al_0\frac{\partial f_{\De(P)}}{\partial \bar z_j}(\bfb),\,\forall j
\]
which says $\bfb$ is a mixed critical point of $f_\De$, a contradiction
 to the strong non-degeneracy. Thus $d(P)+m=0$ and 
\begin{eqnarray*}
b_j&=&\al_0\frac{\overline{\partial f_{\De(P)}}}{\partial  z_j}(\bfb)+\bar\al_0\frac{\partial f_{\De(P)}}{\partial \bar z_j}(\bfb),\,j\in I\\
0&=&\al_0\frac{\overline{\partial f_{\De(P)}}}{\partial z_j}(\bfb)+\bar \al_0\frac{\partial f_{\De(P)}}{\partial \bar z_j}(\bfb),\,j\in K-I.
\end{eqnarray*}
This says that the function $\rho_\De$ has a critical value $\|\bfb_I\|^2$ on $V_\De^*$,
as the gradient vector of $\rho_\De$ is given as
\[
\bar\partial \rho_\De(\bfb)=2 \bfb_I,\quad \bfb_I=\begin{cases} b_i,\,\,&i\in I,\\ 0\,\,& i\notin I.
\end{cases}
\]
Here $V_\De^*=\{\bfz\in \mathbb C^{*n}\,|\, f_\De(\bfz)=0\}$.
This is a contradiction on the assumption
that the function $\rho_\De$ has no critical value on the interval $(0,\rho_0]$, as  $\|\bfb\|\le \rho_0\le r_\De/2$.
\end{proof}
\begin{Remark}
The assertion (2) also follows from $a_f$-condition (see Proposition \ref{af->nearby} below.)
\end{Remark}

\subsection{Tubular Milnor fibration}
Put 
\[\begin{split}
&D(\de_0)^*=\{\eta\in \BC\,|\,0<|\eta|\le \de_0\},\,\,
S^1_{\de_0}=\partial D(\de_0)^*=\{\eta\in \BC\,|\,|\eta|=\de_0\}\\
&E(r,\de_0)^*=f\inv(D(\de_0)^*)\cap B_r^{2n},\,\,
\partial E(r,\de_0)^*=f\inv(S_{\de_0}^1)\cap B_r^{2n}.
\end{split}
\]
By Lemma \ref{MilnorFibering2} and the theorem of Ehresman
(\cite{Wolf1}), we obtain the following  description of the   tubular Milnor fibration  (i.e., the Milnor fibration
 of the second type)
(\cite{Hamm-Le1}).
\begin{Theorem}{\rm(Tubular Milnor fibration)}\label{Milnor2}
Assume that $f(\bfz,\bar\bfz)$ is a strongly non-degenerate
mixed function which is locally tame along the vanishing doordinate subspaces. 
Take positive numbers $r\le \rho_0$
and $\de_0\le \de(r)$ as in
 Lemma
\ref{mixed Hamm-Le1}.
Then $f: E(r,\de_0)^*\to D(\de_0)^*$ and 
$f:\partial E(r,\de_0)^*\to S_{\de_0}^1$ are  locally trivial fibrations
and the topological isomorphism class does not depend on the choice of 
$\de_0$ and $r$.
\end{Theorem}
\subsection{Spherical Milnor fibration}
Consider the spherical Milnor fibration (i.e., Milnor fibration of the first kind):
\[
f/|f|:S_r-K\to S^1,\quad K=V\cap S_r.\]
In the proof of the existence of the spherical fibration and the equivalence to the tubular  Milnor fibration
(Theorem 52, \cite{OkaMix}), we have assumed ``super strongly
non-degeneracy''. However this assumption is used only to prove the
Hamm-L\^e type assertion (Lemma 51, \cite{OkaMix}). We have proved this
Lemma with locally tameness assumption (Lemma \ref{mixed Hamm-Le1}).
Thus we get 
\begin{Theorem} Assume that $f$ is a strongly non-degenerate mixed
 function which is locally tame along vanishing coordinate subspaces.
 For a sufficiently small $r$, the spherical and tubular Minor
 fibrations
 exist and they are  equivalent 
each other.\end{Theorem}
\section{Boundary stability,  $A_f$-condition and transversality of the nearby fibers}
In this section, we consider further geometric properties about mixed polynomials.
\subsection{$a_f$-condition}
Assume that $f$ is a mixed polynomial and  we assume that a Whitney regular stratification $\mathcal S$ of $\mathbb C^n$ is given so that $V=f\inv(0)$ is a union of strata
$M\subset V$. We says that {\em
$f$ satisfies Thom's $a_f$-condition with respect to $\mathcal S$} (locally at $\bf 0$) if there exist positive number $r$ and $\de\ll r$ such that 
$V_\eta=f\inv(\eta)$  with $\eta\ne 0,\,|\eta|\le \de$ is smooth in $B_r$ and any sequence $\bfz^{(\nu)}$ which converges to some $\bfw\ne {\bf 0}$,
$\bfw\in M$, where $M$ is a stratum in $V\cap \mathcal S$ and the tangent space $T_{\bfz^(\nu)}f\inv(f(\bfz^{(\nu)})$ converges to some $\tau$ in the suitable Grassmanian space. Then $T_{\bfw}M$ is a subspace of $\tau$.
The following says that the nearby fiber's transversality follows from $a_f$-condition.
\begin{Proposition}\label{af->nearby}
Assume that $f$ satisfies $a_f$ condition at $\bf 0$ and the nearby
 fibers are smooth. Then there exists a $r_0>0$ such that for any
 $0<r_1\le r_0$, there exists a positive $\de$ so that any nearby fiber $V_\eta$ intersects transversely with the sphere $S_r$ for $r_1\le r\le r_0$ and $0<|\eta|\le \de$.
\end{Proposition}
\begin{proof}
Take $r_0$ so that for any $r\le r_0$, the sphere $S_{r}$ intersects transversely with all strata $M\subset V$.
Note that $M$ and $S_r$ intersect transversely if and only if for any $\bfa\in M\cap S_r$, $T_{\bfa} M$ and $T_{\bfa} S_r$ intersect transversely.
That is $T_{\bfa} M\not \subset T_{\bfa} S_r$.
Take a sequence of points $\bfz^{(\nu)}$ converging to  $\bfa\in M\subset V$ where $M$ is a stratum and $\bfa\ne 0$. Put $\eta_\nu=f(\bfz^{(\nu)})$
and $r_\nu=\|\bfz^{(\nu)}\|$ and $r':=\|\bfa\|$, $r_0\ge r'\ge r_1$.
Assume that $V_{\eta_\nu}$ intersects $S_{r_\nu}$ non-transversely  at $\bfz^{\nu}$.
Then this implies $T_{\bfz^{(\nu)}}f\inv(\eta_\nu)\subset T_{\bfz^{(\nu)}}S_{r_\nu}$. Assume that $T_{\bfz^{(\nu)}}f\inv(\eta_\nu)$ converges to $\tau$.
Then $\tau\subset T_{\bfa}S_{r'}$. On the other hand, $a_f$ condition says that 
$T_{\bfa} M\subset \tau$ and $T_{\bfa}M\not\subset T_{\bfa} S_{r'}$. This is a contradiction.
\end{proof}
\subsection{Boundary stability condition}
Assume that $r_0>0$ is chosen so that $\vphi=f/|f|:S_r\setminus K\to S^1$ is a  fibration
for any $r\le r_0$. We wish to consider the boundary condition
$\overline F_\theta\supset K$ or not. 
This property is always true for holomorphic functions but not always true for mixed functions.
For the argument's simplicity, we consider as follows.
Consider the Milnor fibration in a open ball:
\begin{eqnarray}\label{openballMilnor}
\vphi_{\leq r}=f/|f|: B_r-V\to S^1,\quad \vphi_{\leq r}(\bfz)=f(\bfz)/|f(\bfz)|
\end{eqnarray}
and put $F_{\theta,\leq r}=\vphi_{\leq r}\inv(e^{i\theta})$. 
To distinguish this fibration with usual Milnor fibration on a sphere, we call this fibration
{\em open ball Milnor fibration.}
\begin{Definition}
We say the open  Milnor fibration
 satisfies {\em the stable boundary  condition} if $\overline{F_{\theta,\leq r}}\supset V$
for any $\theta$.
Note that the Milnor fibration in a ball is homotopically equivalent to the one on a fixed sphere
$f/|f|: S_r \setminus K\to S^1$.
\end{Definition}
Recall that a continuous mapping $\vphi:X\to Y$ is {\em an open mapping along  a subset $A\subset X$}
if for any point $a\in A$ and any open neighborhood $U$ of $a$ in $X$, $\vphi(U)$ is a neighborhood of $\vphi(a)$ in  $Y$.
The following  is an immediate consequence of the definition.
\begin{Proposition} The next two conditions are equivalent.
\begin{enumerate}
\item
The boundary stability condition for the Milnor fibration of $f$ is satisfied.
\item  $f:\mathbb C^n\to \mathbb C$ is an open mapping 
along $V\cap B_r$  for a sufficiently small $r>0$. 
\end{enumerate}
 In particular, if $f$ is a holomorphic function,
it satisfies  the boundary stability condition.
\end{Proposition}


\begin{Lemma}\label{nicely} Assume that $f(\bfz,\bar\bfz)$ is a strongly
 non-degenerate and  locally tame along vanishing coordinate subspaces. Then the  Milnor fiibration satisfies the stable boundary property.
\end{Lemma}
\begin{proof}
Take a point $\bfa=(a_1,\dots, a_n)\in V\cap \Int(B_r)$ 
and put $I=\{i\,|\, a_i\ne 0\}$. 

(i) Assume that $I\in \mathcal{I}_{nv}(f)$ so that $\bfa$ is a non-singular point of $V^{*I}$. Then it is obvious that $\bfa\in \overline{F_\theta}$, as $\{V_\eta\},\,|\eta|\le \de\ll r$ is a transversal family with the spheres $S_{r'}$  for $\|\bfa\|/2\le r'\le r$ and 
$V_\eta\subset F_{\theta,\leq r}$ for $\eta,\,\arg\,\eta=\theta$.

(ii)
Assume that $f^I\equiv 0$. Take an essential non-compact face $\De=\De(P)$ with $I(\De)=I$ and consider the face function
$f_P(\bfz,\bar\bfz)$. Put $f_{P,a_I}$ be the restriction of $f_P$ on $z_i=\bfa_i,\,i\in I$.
Thus we consider the polynomial mapping $f_{P,a_I}:\mathbb C^{ n-|I|}\to \mathbb C$.
As $f_{P,\bfa_I}$ is a strongly non-degenerate function
for sufficiently small $\bfa_I$, there exists $\bfb=(b_j)_{j\notin I}$
such that $f_{P,\bfa_I}(\bfb)=\rho e^{i\theta}$ for some $\rho$. Take an arc $\bfb(s),-\eps\le s\le \eps$ so that 
$f_{P,\bfa_I}(\bfb(s))=\rho e^{i(\theta+s)}$ and $\bfb(0)=\bfb$. This is possible as $f_{P,\bfa_I}:\mathbb C^{ n-|I|}\to \mathbb C$
is a submersion.
Consider the path: 
\[
(t,s)\mapsto \bfb(t,s)=(b_j(t,s))_{j=1}^n,\,\quad  b_j(t,s)=
\begin{cases} b_j(s)t^{p_j},\quad &j\notin I\\
a_j,\quad &j\in I.
\end{cases}
\]
Then we have 
\begin{eqnarray*}
f(\bfb(t,s)) &=&f_{P,\bfa_I}(\bfb(s))t^{d(P)}+\text{(higher terms)}\\
           &=&\rho e^{i(\theta+s)}t^{d(P)}+\text{(higher terms)}.
\end{eqnarray*}
Take a sequence $t_\nu\to 0$. As the $\arg\, f(\bfb(t_\nu,s))\to \theta+s$,we can take a sequence 
$s_\nu,\,-\eps\le s_\nu\le \eps$ such that $\arg\, f(\bfb(t_\nu,s_\nu))=\theta$ for sufficiently small $|t_\nu|$.
For example, assume that  $\arg\,f(\bfb (t,0))<\theta$. Note that  $\arg\, f(b(t,\eps))>\theta$ as long as $t\ll 1$. Thus we use the mean value theorem to chose such a $s_\nu$.
The point $\bfb(t_\nu,s_\nu)\in F_{\theta,\leq r}$ for sufficiently small $|t_\nu|$
and it converges to $\bfa$. This implies that the closure of $F_{\theta,\leq r}$ contains $V$.
\end{proof}

\subsection{Strongly non-degenerate polynomials which is not locally tame}

(1) {\bf Example 1}. Consider the example of M. Tibar:  $f(\bfz)=z_1|z_2|^2$ (\cite{TibarOberwolfach,SCT1,SCT2}).
This is a mixed weighted homogeneous polynomial. Thus it is strongly non-degenerate. A polar weight can be $P={}^t(1,0)$. $S^1$-action is defined 
as $\rho\circ (z_1,z_2)=(z_1\rho, z_2)$ for $\rho\in S^1$.
Then for any $r>0$, there exists a spherical Milnor fibration:
$\vphi=f/|f|:S_r\setminus K\to S^1$. 

{ First we show that the boundary stability is not satisfied.}
Take a fiber $F_\theta$.
$K$ has two components, $K_1=\{z_1=0\}$ and $K_2=\{z_2=0\}$. 
The closure of $F_\theta$ is given as 
$\bar F_\theta=F_\theta\cup  K_1\cup \{(re^{i\theta},0)\}$.
Thus the intersection $\bar F_\theta\cap K_2$ is a single point $(re^{i\theta},0)$
and this point $(re^{i\theta},0)$ turns along $K_2$ once  when $\theta$ goes from $0$ to $2\pi$.
Note that $K_2$ is a $S^1$-orbit of the action. We call $K_2$ {\em a rotating axis}.
The function $f$ is not locally tame along the vanishing  axis $z_2=0$  by Lemma \ref{nicely}.
In fact, take a point $(a,0)\in K_2$ and put $a=\rho e^{i\theta}$. Take an open set 
$U=\{z_1\,|\, |z_1-a|\le \eps\}\times \{z_2\,|\, |z_2|\le \eps\}$ and  put $\al$ be the small positive angle so that 
$\tan \al=\eps/\rho$.
Then the image of $U$ by $f$ is  contained in the angular region
$\{\eta\in \mathbb C\,|\, \theta-\al \le \arg\,\eta\le \theta+\al\}$. Thus it is not an open mapping.
More precisely we assert 
\begin{Assertion}
 $\overline{F_\theta}$ is homeomorphic to $\Cone(K_1)$. 
\end{Assertion}
For example, taking $r=1$, consider the mapping $\psi: \overline{F_\theta}\to \Cone(K_1)$, defined by
$\psi(z_1,z_2)=(1-|z_1|, \arg(z_2))$. Here  we understand 
\[\Cone (K_2)=[0,1]\times K_2/\{0\}\times K_2,\quad K_2\simeq S^1.
\]

M. Tibar observed that $f$ does not have any stratification which  satisfies  the  $a_f$ condition
along $z_1$ axis (\cite{Pichon-Seade}). Put $f=g+ih$ with $g=x_1(x_2^2+y_2^2)$ and $h=y_1(x_2^2+y_2^2)$.
Then the Jacobian matrix is given as
\[
J(g,h)=\left(\begin{matrix}
x_2^2+y_2^2&0&2a_1x_2&2a_1y_2\\
0&x_2^2+y_2^2&2b_1x_2&2b_1y_2\\
\end{matrix}
\right)
\]
Note that the last $2\times 2$ minor has rank one and this makes the
problem at the limit.   Take a point $p=(a_1+ib_1,0)$.
Consider the rotated mixed polynomial
$\tilde f:=(b_1+a_1i)f$ and write it as $\tilde f= \tilde g+i\tilde h$. 
Note that $f\inv(f(p))=\tilde f\inv(\tilde f (p))$ and 
$\tilde g=b_1g-a_1h$. Then the normalized gradient of $\tilde g$
is given by 
\[
\grad\, \tilde g=(b_1,-a_1,0,0).
\]
Put $p=(a_1+b_1 i,z_2)$.
Thus when $z_2\to 0$, 
\[T_pf\inv(f(p))\subset T_p\tilde g\inv(\tilde g(p))\not \supset \mathbb C\times \{0\}.
\]
This implies, if there is a stratification which satisfies
$a_f$-condition, the stratum of $\mathbb C\times \{0\}$ which contains
$p$ can not be two dimensional at $p\in \{z_2=0\}$. As this is the case at any point of $\{z_2=0\}$,
there does not exist any stratification which satisfies $a_f$ condition.
On the other hand, we assert that 
\begin{Proposition}$f$ satisfies the transversality condition for the nearby fibers.
\end{Proposition}
\begin{proof} We may assume that the sphere has radius 1, by the polar homogenuity.
Assume that there is a sequence $p_\nu=(u_\nu,v_\nu)\in S_1^3$ such that $f\inv(f(p_\nu))$ is not transverse to $S_1^3$ and $f(p_\nu)\to 0$.
Then either $u_\nu\to 0$ or $v_\nu\to 0$ ( equivalently either $|v_\nu|\to 1$ or $|u_\nu|\to 1$).
We may assume that $p_\nu=\al_\nu \overline{\partial f}+\bar\al\bar\partial f$ by Lemma 2 which is equivalent to
\begin{eqnarray*}
\begin{cases}
u_\nu&=\al_\nu |v_\nu|^2\\
v_\nu &=\al_\nu\bar u_\nu v_\nu+\bar\al_\nu u_\nu v_\nu.
\end{cases}
\end{eqnarray*}
From the first equality, we can put
$u_\nu=r_\nu e^{i \theta_\nu},\, \al=\rho_\nu e^{i \theta_\nu}$.
The second equality says that $1=2 \rho_\nu r_\nu$. Thus $\rho_\nu\to 1/2$ if $r_\nu\to 1$
which implies $|v_\nu|\to 2$ and $|f(p_\nu)|\not\to 0$.
Assume that $r_\nu\to 0$. Then $|v_\nu|^2=r_\nu/\rho_\nu=2r_\nu^2\to 0$.
This is also impossible, as $|p_\nu|=1$.
\end{proof}
This example shows that the transversality of nearby fibers does not implies  either tameness or $a_f$-condition.
On the other hand, tameness with strong non-degeneracy implies transversality of the nearby fibers, as we will see below.

{\bf (2)  Example of A. Parusinski}: $f=z_1(z_2+z_3^2)\bar z_2$ (\cite{Pichon-Seade},see also \cite{SCT1,SCT2}). 
Note that 
$f$ is strongly  non-degenerate.

\begin{Proposition} (A. Parusinski) Consider $I=\{1\}$ and note that $f|{\mathbb C^I}\equiv 0$.
Then $f$ does not satisfies $a_f$-condition along $z_1$-axis $\{z_2=z_3=0\}$.
\end{Proposition}
\begin{proof}
The proof goes in the same line as that in Example 1.
Consider the weight $P={}^t(0,1,3)$.
Then $f_P=z_1|z_2|^2$ and $d(P)=2$. Assume that there exists a stratification $\mathcal S$
satisfying $a_f$-condition. We show the contradiction.
Take a point $p=(re^{i\theta},0,0)$ and assume that $p\in M$ where $M$ is a real two dimensional stratum
of $\mathbb C^I$.
Consider the modified function
$\tilde f=(\sin\theta+i\cos\theta)f$. Then the real part $\tilde g$ of $\tilde f$
is given as 
\begin{eqnarray*}\tilde g &=&
\sin\theta\, g-\cos\theta\, h\\
&=& (x_1\sin\theta - y_1\cos\theta ) |z_2|^2 +\Re\, (e^{i(\pi/2-\theta)}z_1\bar z_2 z_3^2)
\end{eqnarray*}
and the gradient vector of $\tilde g$ at $\bfz(t):=(p,ta_2,t^3a_3)$
for $a_2,a_3\in \mathbb C^*$ fixed  is given as 
\begin{eqnarray*}
\grad\,\tilde g(p,ta_2,t^3a_3)&=&(\sin\theta, -\cos\theta,0,0,0,0)|a_2|^2t^2\\
&&+O(t^3).
\end{eqnarray*}
Thus the normalized gradient vector converges to
\[\bfv:=(\sin{\theta},-\cos\theta,0,0,0,0).\]
This implies that 
\begin{multline*}
T_{\bfz(t)}f\inv(f(\bfz(t))\subset T_{\bfz(t)}\tilde g\inv(\tilde g(\bfz(t))\,\,\mapright{t\to 0}\,\,
\bfv^{\perp}\not\supset
\mathbb C^I.\end{multline*}
This is a contradiction.
\end{proof}

\begin{Remark}We do not know (and do not care) if $f\inv(\eta),\eta\ne 0$ is  a transverse family for sufficiently small $\eta$. 
\end{Remark}
{\bf (3) Example 3.} Consider 
\[
f(\bfz,\bar\bfz)=z_1 k(\bfz),\quad k(\bfz):=\sum_{i=1}^m|z_i|^{2a_i}-\sum_{j=m+1}^n|z_j|^{2a_j}
\]
for $2\le m<n$. Then $f$ is not strongly non-degenerate but polar weighted homogeneous and it has a Milnor fibration.
However it is not locally tame along the vanishing coordinate subspaces and $f$ does not satisfy the $a_f$-condition.
In fact the link has two  components $K_1=\{z_1=0\}$ and $K_2=\{k(\bfz)=0\}$. The component $K_2$ has {\bf real codimension 1} and at any point of $K_2\setminus K_1$,
$f$ is not open mapping and thus 
\[
\overline{F_\theta}=F_\theta\cup K_1\cup \{\bfz\in S_r\,|\, \arg\,z_1=\pm\theta\}
\]
where sign is the same as that of $k(\bfz)$.
Thus $K_2$ is a rotation axis. The monodromy is the rotation arround $z_1$ axis:
\[
h_\theta:F_0\to F_\theta,\quad (z_1,\bfz')\mapsto (z_1e^{i\theta},\bfz').
\]
 The fiber $F_\theta$ has two components, 
$F_\theta^+=\{\arg\,z_1=\theta, k(\bfz)>0\}$ and $F_\theta^-=\{\arg\,z_1=-\theta,k(\bfz)<0\}$.
\begin{Remark}The function $k(\bfz)$ is a real valued polynomial and the fibers $k\inv(\eta)$ are smooth for $\eta\ne 0$ and
$k\inv(0)$ has an isolated singularity as a real hypersurface. However as a mixed function $k:\mathbb C^n\to \mathbb C$,
it has no regular points.
\end{Remark}
\subsection{Thom's $a_f$-condition}
By analyzing above  examples,  we notice that 
the limit of two independent hyperplanes $T_{p}g\inv(g(p))$ and $T_p h\inv(h(p))$ may not independent
 when $p$ goes to some point of vanishing coordinate $\mathbb C^I$, and  this phenomena induces  a failure of $a_f$ condition.
This problem does not occur under the tameness condition.

\begin{Theorem}\label{a_f-nice} Assume that $f(\bfz)$  is a strongly
 non-degenerate polynomial and assume that $f$ is locally tame along
vanishing coordinate subspaces. 
We consider the canonical  stratification $\mathcal S_{can}$  which is defined  by
\[
\mathcal S_{can}:\quad \{V\cap \mathbb C^{*I},\, \mathbb C^{*I}\setminus V\cap \mathbb C^I\,|\, I\in \mathcal{I}_{nv}(f)\}\cup\{\mathbb C^{*I}\,|\, I\in \mathcal{I}_v(f)\}.
\]
Then $f$ satisfies $a_f$-condition with respect to $\mathcal S_{can}$.
\end{Theorem}
\begin{proof}
Take a point $\bfq^I=( q_j)_{j\in I}\in V\cap \mathbb C^{*I}$.
Using Curve Selection Lemma,  it is enough to check the $a_f$-condition along an arbitrary analytic path. So
take any analytic path $\bfz(t)$ such that 
$\bfz(0)=\bfq^I$ and $\bfz(t)\in \mathbb C^{*J}$  for $t\ne 0$ with 
$I\subset J$ with $I\ne J$.
As the argument is precisely the same, we assume hereafter  that $J=\{1,\dots, n\}$.
We will show that $a_f$-condition is satisfied 
for this curve. By non-degeneracy, we may assume that $I\in \mathcal I_v(f)$ so that $\mathbb C^I$ is a vanishing coordinate. (Otherwise, $p$ is a smooth point of $V$ and the $a_f$-condition is obviously satisfied.)
Consider the Taylor expansion:
\begin{eqnarray*}
z_j(t)=a_jt^{p_j}+\text{(higher terms)}, \quad
\begin{cases} p_j=0,a_j=q_j,\quad &j\in I\\
p_j>0,\quad & j\not\in I.
\end{cases}
\end{eqnarray*}
Put $P={}^t(p_1,\dots, p_n)$, $\bfa=(a_1,\dots, a_n)$, $d=d(P)$ and $\De=\De(P)$. For notation's simplicity, we assume that 
$I=\{m+1,\dots, n\}$.
Note that 
\[\begin{split}
&\frac{\partial g}{\partial\bar z_j}(\bfz(t))=\frac{\partial g_\De}{\partial\bar z_j}(\bfa) t^{d-p_j}+\text{(higher terms)}\\
&\frac{\partial h}{\partial\bar z_j}(\bfz(t))=\frac{\partial h_\De}{\partial\bar z_j}(\bfa) t^{d-p_j}+\text{(higher terms)}.
\end{split}
\]

For simplicity, we assume that 
$p_1\ge p_2\ge \dots\ge p_m$. 
For a vector $v=(v_1,\dots, v_n)$ and $1\le \al\le \be\le m$, we  consider the truncation
\[
v_\al^\be:=(v_\al,\dots, v_\be).
\]
We choose $1\le \al\le \be\le m$ as follows.

(A-1)
 For any $j<\al$, $\frac{\partial}{\partial \bar z_j} g_\De(\bfa)=0$,
$\frac{\partial}{\partial \bar z_j} h_\De(\bfa)=0$ and
\nl \indent
$(\frac{\partial}{\partial \bar z_\al} g_\De(\bfa),\frac{\partial}{\partial \bar z_\al} h_\De(\bfa))\ne (0,0)$.

(A-2) Two complex vectors 
\[\begin{split}
(\bar\partial g_\De(\bfa))_\al^{\be}=
&(\frac{\partial}{\partial\bar z_\al} g_\De(\bfa),\cdots,\frac{\partial}{\partial \bar z_{\be}} g_\De(\bfa))\\
(\bar\partial h_\De(\bfa))_\al^{\be}=&(\frac{\partial}{\partial\bar z_\al} h_\De(\bfa),\cdots,\frac{\partial}{\partial\bar z_{\be}} h_\De(\bfa))
\end{split}
\]
are linearly independent over $\mathbb R$ 
and $(\bar\partial g_\De(\bfa))_\al^{\be'},\,(\bar\partial h_\De(\bfa))_\al^{\be'}$ are linearly dependent over $\mathbb R$ for any 
$\be'<\be$.
For simplicity, we use the notations:
\[
v_g(t):=\bar\partial g(\bfz(t))=(v_{g,1},\dots, v_{g,m}),\quad v_h(t):=\bar\partial h(\bfz(t))=(v_{h,1},\dots, v_{g,m}).
\]
We consider the order of $v_g(t)=\bar\partial g(\bfz(t))$ and $v_h(t)=\bar \partial h(\bfz(t))$.
(Here the order is the lowest degree of in $t$.)

Suppose $\ord\,v_g=r$ and the smallest index  $1\le i\le m$ with
 $\ord\,v_{g,i}=r$ is called {\em leading index}. Assume that $s$ is the
 leading index of $v_g(t)$. We call the coefficient of $t^r$ in the
 expansion of
$v_{g,s}(t)$ {\em the leading coefficient}. Put 
$s'$ be the leading index of $v_h$.

For simplicity, we assume that
$s\le s'$ and if $s=s'$ we assume also $\ord\, v_g(t)\le \ord\, v_h(t)$. This is possible by changing $g$ and $h$ considering 
$if(\bfz,\bar\bfz)$, if necessary.

First we observe that
\[\ord\,v_{g,i}(t), \ord\,v_{h,i}(t)\ge d-p_i,\quad s\le d-p_\al.
\]

\vspace{.2cm}
{\bf Strategy.} Put $r=\ord\, v_g(t),\,r'=\ord\, v_h(t)$. We have three possible cases.

(1) $s'>s$ or

(2-a) $s=s'$ and the coefficients of $t^r$ of $v_{g,s}$ and
the coefficient of $t^{r'}$ of  $v_{h,s}$ are linearly independent over $\mathbb R$
or 

(2-b) $s=s'$ and the coefficients of $t^r$ of $v_{g,s}$ and
the coefficient of $t^{r'}$ of  $v_{h,s}$ are linearly dependent over $\mathbb R$.

For (1) or (1-a),
we have  nothing to do. In fact,
write
\[\begin{split}
&v_g(t)=v_g^\infty\infty t^r+\text{(higher terms)},\quad v_g^\infty\in \mathbb C^{n}\\
&v_h(t)=v_h^{\infty} t^{r'}+\text{(higher terms)},\quad v_h^\infty\in \mathbb C^{n}.
\end{split}
\]
Then the normalized limit of $v_g(t), v_h(t)$ are given by   $v_g^\infty/\|v_g^\infty\|,\, v_h^{\infty}/\|v_h^{\infty}\|$.
In this case, the limit of $v_g$  and $v_h$ for $t\to 0$ are complex  vectors $v_{g}^\infty,\,v_h^\infty$
(up to scalar multiplications)  which are  in $\mathbb C^m\times \{0\}$.
They are linearly independent over $\mathbb R$. Thus the limit of $T_{\bfz(t)}f\inv(f(\bfz(t))$ is  the real 
orthogonal complement $<v_g^\infty,v_h^{\infty}>^{\perp}={v_g^\infty}^\perp\cap{v_h^\infty}^\perp$
which contains $ \mathbb C^I$.

\vspace{.2cm}
Assume  $s=s'$ and the coefficients of $t^r$ in  $v_{g,s}$ and
the coefficient of  $t^{r'}$ in  $v_{h,s}$ are linearly dependent over $\mathbb R$.
Then we consider the following operation. 

\vspace{.2cm}
{\bf Operation.} Put $r'=\ord\, v_h$. We have assumed $r'\ge r$.
Take a unique real number $\la$ and replace $v_h$ by
$v_h'=v_h-\la t^{r'-r}v_g$ with $r=\ord\, v_{g,s},\,r'=\ord\, v_{h,s}$ to kill the coefficient of $t^{r'}$ of $v_{h,s}$.
(We have assumed $r\le r'$.)

Note that after this operation,
the vector $v_{h,j}'$ changes into  
\[v_{h,j}'(t)=(\frac{\partial}{\partial \bar z_j}h_\De(\bfa)-\la\eps\frac {\partial}{\partial \bar z_j} g_\De(\bfa))t^{d-p_j}+\text{(higher terms)}
\]
where $\eps=1$ or $0$ according to $r'=r$ or $r'>r$ respectively.
We observe that if $r'>r$, the leading term of $v_{h,j}'(t)$ does not change.  If $r'=r$,
the (leading) coefficient 
$\frac{\partial}{\partial \bar z_j}h_\De(\bfa)$
 of $t^{d-p_j}$ in $v_{h,j}$ is changed into
\[
 \frac{\partial}{\partial \bar z_j}h_\De(\bfa)-\la\frac
 {\partial}{\partial \bar z_j} g_\De(\bfa),
\]
 the above two properties (A-1), ( A-2)
 are
unchanged.  

We continue the operation as long as 
the leading index of $v_h'$ is still $s$. Suppose that this operation stops at $k$-th step.
Then 
put $s^{(k)}$ the leading index  of $v_h^{(k)}$ and $r^{(k)}$ be the order of $v_h^{(k)}$. 
By the above two properties, $s^{(k)}\le \be$ and $r^{(k)}\le d-p_\be$.
This implies that the  limit of the normalized  gradient vectors
$v_g$ and $v_{h}^{(k)}$, say $v_g^\infty, v_h^\infty$  are independent vectors in $\mathbb C^m\times \{0\}=\mathbb C^{I^c}$
over $\mathbb R$.
On the other hand, by the definition of the above operations,
\[\begin{split}
T_{\bfz(t)} f\inv(f(\bfz(t))&=v_g(t)^{\perp}\cap v_h(t)^{\perp}\\
&=v_g(t)^{\perp}\cap v_h'(t)^{\perp}=\cdots=v_g(t)^{\perp}\cap (v_h^{(k)}(t))^{\perp}.
\end{split}
\]
Thus the limit of $T_{\bfz(t)} f\inv(f(\bfz(t))$ is nothing but $(v_g^\infty)^{\perp}\cap(v_h^{\infty})^{\perp}$.
Note that $(v_g^\infty)^{\perp}\cap(v_h^{\infty})^{\perp}\supset
\mathbb C^{I}$.
This
 show that the $a_f$-property is satisfied along this curve.
\end{proof}
The following will be practically useful.
\begin{Lemma}\label{niceness-check} 
Let $f_\De$ be a face function associated with an essential  non-compact face $\De\in\Ga_{nc}(f)$ with $I=I(\De)$.
Assume that $I=\{m+1,\dots, n\}$.   
\begin{enumerate} 
\item  For $f(\bfz)$  a holomorphic function,  the following is necessary and sufficient for $f_\De$ to be locally tame.
\[(\frac{\partial}{\partial z_1}f_\De(\bfz),\dots, \frac{\partial}{\partial z_m}f_\De(\bfz))
\]
is  a non-zero
vector for any $\bfz$ with $\|\bfz_I\|\le \rho_0$.
\item For a mixed polynomial,  $f_\De$ is locally tame 
if  there exists a $ j\in I^c$
such that two complex numbers
$\frac{\partial g_\De}{\partial \bar z_j}(\bfz),\frac{\partial h_\De}{\partial \bar z_j}(\bfz)$ are linearly independent over $\mathbb R$.
In other word, 
\[\Im\,(  \frac{\partial g_\De}{\partial \bar z_j}(\bfz)\overline{\frac{\partial h_\De}{\partial \bar z_j}(\bfz) }  )\ne 0.
\]
 for  any $\bfz$ with $\|\bfz_I\|\le \rho_0$.
\end{enumerate}
\end{Lemma}
\begin{proof}Recall that
\[
\bar \partial g=\frac 12 (\bar\partial f+\overline{\partial f}),\quad
\bar\partial h=\frac{i}{2}(\bar\partial f-\overline{\partial f}).
\]
If $f$ is holomorphic, $\bar \partial g=\frac12 \partial f$ and $\bar\partial h_\De=-i\bar\partial g_\De$ and they are perpendicular by the Euclidean inner product.
Thus they are independent over $\mathbb R$.
For the second assertion, note that
the assumption is equivalent to the $2\time 2$ minor
\[
\det\,\left(
\begin{matrix}
\frac{\partial g_\De}{\partial x_j}(\bfa)&\frac{\partial g_\De}{\partial y_j}(\bfa)\\
\frac{\partial h_\De}{\partial x_j}(\bfa)&\frac{\partial h_\De}{\partial y_j}(\bfa)
\end{matrix}
\right)=-\Im\,(  \frac{\partial g_\De}{\partial \bar z_j}(\bfa)\overline{\frac{\partial h_\De}{\partial \bar z_j}(\bfa) }  )
\ne 0.
\]
\end{proof}

\subsubsection{Examples}
Example 1. (Modification of Tibar's example) Consider the mixed monomial $f=z_1z_2^a \bar z_2$.
Then we have 
\[\begin{split}
&\bar\partial f=(0,z_1z_2^a),\quad \overline{\partial f}=(\bar z_2^a z_2,a \bar z_1\bar z_2^{a-1} z_2)\\
&\bar\partial g=\frac 12 (\bar z_2^a z_2,z_1z_2^a+a \bar z_1\bar z_2^{a-1} z_2)\\
&\bar \partial h=\frac i2 (-\bar z_2^a z_2,z_1z_2^a-a \bar z_1\bar z_2^{a-1} z_2)
\end{split}
\]
Consider the vanishing coordinate $I=\{1\}$. Two complex numbers 
\[z_1z_2^a+a \bar z_1\bar z_2^{a-1} z_2,\,\quad i(z_1z_2^a-a \bar z_1\bar z_2^{a-1} z_2)\]
 are 
linearly dependent over $\mathbb R$ if and 
only if $a=1$ as
\[\begin{split}
&(z_1z_2^a+a \bar z_1\bar z_2^{a-1} z_2)(-i)(\bar z_1\bar z_2^a-a  z_1 z_2^{a-1} \bar z_2)\\
&=
-i (1-a^2) |z_1|^2|z_2|^{2a} -i  a(- z_1^2z_2^{2a-1}\bar z_2+\bar z_1^2 \bar z_2^{2a-1} z_2)\\
&=-i(1-a^2) |z_1|^2|z_2|^{2a} -2a\Im ( z_1^2z_2^{2a-1}\bar z_2  ) 
\end{split}
\]
Thus the imaginary part of the above complex number is zero if and only if $a=1$.
Note that $f$ is an open mapping along $z_2=0$ if and only if $a>1$.

Example 2. Consider the mixed polynomial
\[
f(\bfz,\bar\bfz)=z_1^{a_1}\bar z_2+z_2^{a_2}\bar z_3+\cdots+z_n^{a_n}\bar z_1,
\quad a_1,\dots, a_n\ge 2.
\]
Then $\mathcal I_v=\{\{i\}\,|\, i=1,\dots, n\}$.
Consider for example, $I=\{n\}$. Then possible face functions are
\[
f_\De=z_2^{a_2}\bar z_3+\cdots+z_n^{a_n}\bar z_1
\]
and its face functions $f_\Xi$ with $\Xi\subset \De$.
\begin{eqnarray*}
\bar\partial f_\De&=&(z_n^{a_n},0,z_2^{a_2},\dots, z_{n-1}^{a_{n-1}})\\
\overline{\partial f_\De}&=&
(0,a_2{\bar z_2^{a_2-1}}z_3,\dots, a_n\bar z_n^{a_n-1}z_1)\\
\bar\partial g_\De&=&\frac 12 (z_n^{a_n},a_2\bar z_2^{a_2-1}z_3, z_2^{a_2}+a_2\bar z_3^{a_3-1}z_4,\dots, z_{n-1}^{a_{n-1}}+a_n\bar z_n^{a_n-1}z_1)\\
\bar\partial h_\De&=&\frac i2 (z_n^{a_n},-a_2\bar z_2^{a_2-1}z_3, z_2^{a_2}-a_2\bar z_3^{a_3-1}z_4,\dots, z_{n-1}^{a_{n-1}}-a_n\bar z_n^{a_n-1}z_1)\\
\end{eqnarray*}
Thus 
\[
(\bar\partial g_\De)_1\cdot (\overline{\bar\partial h_\De})_1=-\frac i4|z_n|^{2a_n}
\]
and its imaginary part is non-zero which satisfies the condition of Lemma \ref{niceness-check}.
Now we consider a subset $\Xi\subset \De$. We consider the first monomial $z_j^{a_j}\bar z_{j+1}$ so that
\[
z_n^{a_n}\bar z_1,\dots, z_{j-1}^{a_{j-1}}\bar z_j\not \in f_\Xi,\quad
z_j^{a_j}\bar z_{j+1}\in f_\Xi.
\]
Then we have
\[
\Im (\bar\partial g_\Xi)_j\cdot (\overline{\bar\partial h_\Xi})_j=-\frac 14 a_k^2|z_k|^{2a_k-2}|z_{k+1}|^2\ne 0.
\]
Thus by symmetry, we conclude that $f$ is locally tame along each vanishing coordinate axis $z_k,\,k=1,\dots, n$.
\section{Some application}
\subsection{Mixed  cyclic coverings}
Consider positive integer vectors  
\[\bfa:=(a_1,\dots, a_n),\,\bfb=(b_1,\dots, b_n)\]
 such that 
$a_j>b_j\ge 0$, $j=1,\dots, n$. We consider the mapping
\[
\vphi_{\bfa,\bfb}:\mathbb C^n\to \mathbb C^n,\quad 
(z_1,\dots, z_n)\mapsto (z_1^{a_1}\bar z_1^{b_1},\dots, z_n^{a_n}\bar z_n^{b_n}).
\]
This is a $\prod_{j=1}^n(a_j-b_j)$-fold multi-cyclic covering branched along the coordinate hyperplanes $\{z_j=0\},\,j=1,\dots, n$.
Consider a holomorphic function $f(\bfz)$ which has a non-degenerate
Newton boundary  
and the pull-back
$\tilde f(\bfz,\bar \bfz):=f(\vphi_{\bfa,\bfb}(\bfz,\bar \bfz))$.
This give a strongly non-degenerate mixed function (\cite{OkaStrong}).

\begin{Proposition}
Assume that $f(\bfz)$ is a non-degenerate holmorphic function which 
is locally tame along their vanishing coordinate subspaces.
Then $\tilde f(\bfw,\bar\bfw):=f(\vphi_{\bfa,\bfb}(\bfw,\bar\bfw))$ is a
 non-degenerate mixed function.  Its vanishing coordinate subspaces are
 the same as that of $f(\bfz)$ and  it is locally tame along the vanishing coordinate subspaces.
\end{Proposition}
\begin{proof}
Take a face function $f_P(\bfz)$ with weight $P={}^t(p_1,\dots, p_n)$.
Consider the weight $\tilde P$ which is the primitive weight vector obtained by multiplying
the least common multiple of the denominators of 
\[
 {}^t(\frac{p_1}{a_1+b_1},\dots, \frac{p_n}{a_n+b_n}).
\]
Then $\vphi_{\bfa,\bfb}^*f(\bfw,\bar\bfw)$ is radially weighted
 homogeneous with respect to the weight $\tilde P$.
We observe also have $\tilde f_{\tilde P}(\bfw,\bar\bfw)=\vphi_{\bfa,\bfb}^*f_P(\bfw,\bar\bfw)$.
Thus we see that the Newton boundary $\Ga(\tilde f)$ corresponds
 bijectively
to that of $\Ga(f)$ by this mapping.
Suppose that $I\in \mathcal I_v(f)=\mathcal I_v(\tilde f)$. 
We assume $I=\{m+1,\dots, n\}$ 
for simplicity.
Take a non-compact face $\De$ with $I(\De)=I$ and let $\tilde \De$ be the corresponding non-compact face of $\tilde f$.
We consider $\tilde f_{\tilde \De}$ as  the following composition, fixing $(u_{m+1},\dots, u_n)\in \mathbb C^{*I}$:

\[
\tilde f_{\tilde \De}:\mathbb C^{*m}\mapright{\vphi_{\bfa,\bfb}'}\mathbb C^{*m}\mapright{f_\De}\mathbb C
\]
where 
\[
\vphi_{\bfa,\bfb}'(w_1,\dots, w_m)=(w_1^{a_1}\bar w_1^{b_1},\dots, w_m^{a_m}\bar w_m^{b_m},u_{m+1}^{a_{m+1}}\bar u_{m+1}^{b_{m+1}},\dots,
u_n^{a_n}\bar u_n^{b_n}).
\]
As $\vphi_{\bfa,\bfb}'$ is a unbranched covering mapping, $\tilde f_{\tilde \De}$ does not have any critical points.
\end{proof}

\subsection{Mixed functions with strongly polar weighted homogeneous faces}
We say a mixed polynomial $h(\bfz,\bar \bfz)$ is {\em mixed weighted homogeneous} if it is radially weighted homogeneous and also  polar weighted homogeneous.  $h(\bfz,\bar \bfz)$ is 
 {\em  strongly polar weighted  homogeneous}
if  the polar weight and the radial weight can be the same.
A mixed function $f(\bfz,\bar\bfz)$ is called of strongly polar weighted homogeneous face type if every face function $f_\De$ is strongly polar weighted homogeneous polynomial (\cite{OkaStrong}).
Let $\Ga^*(f)$ be the Newton boundary and let $\Si^*$ be an admissible  regular subdivision of $\Ga^*(f)$ and let $\hat \pi:X\to \mathbb C^n$
be the associated toric modification. Let $\mathcal V$ be the vertices of $\Si^*$ which corresponds to the exceptional divisors as in \S 2, \cite{OkaStrong}.
Let $\mathcal S_I$ be the set of $|I|-1$ dimensional faces of $\Ga(f^I)$.
It is shown that $\hat \pi:X\to \mathbb C^n$  topologically resolve the mixed function $f:\mathbb C^n\to \mathbb C$ (\cite{OkaStrong}).
Combining the existence of Milnor fibration and the argument in \cite{OkaStrong}, we can generalize Theorem 11 (\cite{OkaStrong})
as follows. For $I\in \mathcal I_{nv}$, we denote by $\mathcal S_I$ the set of weight vectors which correspond to 
$|I|-1$ dimensional faces of $\Ga(f^I)$.

The notations and definitions are the same as in Theorem 11 (\cite{OkaStrong}).
\begin{Theorem} 
Let $f(\bfz,\bar\bfz)$ a  non-degenerate
mixed polynomial of strongly polar positive weighted homogeneous face
 type which is locally tame  along vanishing 
coordinate subspaces.
 Let $V=f\inv(V)$ be a germ of hypersurface 
at the origin and let $ \widetilde V$ be the strict transform of $V$ to $X$.
Then

\noindent
(1) $\tilde V$ is topologically smooth and  real  analytic 
 smooth variety outside of the union of the exceptional divisors $\cup_{P\in \mathcal V}\hat E(P)$.


\noindent
(2)
The zeta function of the Milnor fibration  of $f(\bfz,\bar\bfz)$
is given by the formula
\[\begin{split}
 &\zeta(t)=\prod_I
\zeta_I(t),\,
 \zeta_I(t)=\prod_{P\in
\cS_I}(1-t^{\pdeg(P,f_P^I)})^{-\chi(P)/\pdeg(P,f_P^I)}\\
\end{split}\]
where $\chi(P)$ is the Euler characteristic of the torus Milnor fiber of $f_P^I$
\[
 F_P^*=\{\bfz_I\in \mathbb C^{*I}\,|\, f^I_P(\bfz_I)=1\},\,\,P\in
 \mathcal S_I.
\]
\end{Theorem}
\subsubsection{Example: 1. Curves with mixed Brieskorn faces}
Consider a mixed polynomial
\[
f(\bfz,\bar\bfz)=z_1^2z_2^2(z_1^6\bar z_1^3+z_2^4\bar z_2^2)(z_1^4\bar z_1^2+z_2^6\bar z_2^3)
\]
$f$ is strongly non-degenerate and has two faces which are strongly polar weighted homogeneous:
Put $P={}^t(2,3),\,Q:={}^t(3,2)$. There are two faces
corresponding to $P$ and $Q$.
\[
f_P(\bfz,\bar\bfz)=z_1^6\bar z_1^2z_2^2(z_1^6\bar z_1^3+z_2^4\bar z_2^2),\,
f_Q(\bfz,\bar\bfz)=z_1^2z_2^6\bar z_2^2(z_1^4\bar z_1^2+z_2^6\bar z_2^3)
\]
and $f_P, f_Q$ are strongly polar weighted with  $\pdeg\, f_P=\pdeg\, f_Q=20$. Thus
the contribution of $f_P$ to the zeta-function is $(1-t^{20})^{-\chi(P)/20}$
where $\chi(P)$ is the Euler characteristic of
\[
F_P^*:=\{\bfz\in \mathbb C^{*2}\,|\, f_P(\bfz,\bar\bfz)=1\}.
\] 
$F_P^*$ is diffeomorphic to
\[
F_P':=\{\bfz\in \mathbb C^{*2}\, |\, z_1^4z_2^2(z_1^3+z_2^2)=1\}
\]
by Theorem 10 (\cite{OkaPolar}). Thus $\chi(P)=\chi(F_P^*)=-20$.
Thus using the symmetry of $f_P$ and $f_Q$, we  get
$\zeta(t)=(1-t^{20})^2$.

In general, for a non-degenerate non-convenient mixed polynomial of two variables $f(\bfz,\bar\bfz)$, consider the right end monomial
$z_1^m\bar z_1^n z_2^a\bar z_2^b$. Right end means that $\Ga(f)$ is in the space $\{(\nu,\mu)\,|\, \nu\le m+n, \mu\ge a+b\}$.
If $a+b\ge 1$, $z_1$ axis is a vanishing coordinate. It is locally tame  along $z_1$-axis if and only if $a-b\ne 0$.

Example 2. Consider $D_n$ singularity:
\[
D_n:\quad f(z_1,z_2,z_3)=z_1^2+z_2^2z_3+z_3^{n-1}.
\]
Then  the Milnor number $\mu(f)$ of $f$ is $n$ and the zeta function is given as 
$\zeta(t)=(t^{n-1}+1)(t^2-1)$.
$f$ has a vanishing axis $z_2$ but $V$ is non-singular except at the origin.
Consider
\[\begin{split}
\tilde f(\bfw,\bar \bfw)&=\vphi_{2,1}^* f(\bfw,\bar\bfw)\\
&=w_1^4\bar w_1^2+w_2^4\bar w_2^2w_3^2\bar w_3+w_3^{2(n-1)}\bar w_3^{n-1}
\end{split}
\]
$\tilde f$ has a vanishing coordinate axis $w_2$  but the data for the zeta function
is exactly same as $f$. As $\vphi_{2,1}$ is a homeomorhism,
$\mu(\tilde f)=n$ and it has the same zeta functions as $f$. See also Corollary 15, \cite{OkaStrong}.
\subsection{Join type polynomials}
We consider the join type polynomial
\[
f(\bfz,\bar\bfz,\bfw,\bar \bfw)=f_1(\bfz,\bar\bfz)+f_2(\bfw,\bar\bfw),\quad (\bfz,\bfw)\in \mathbb C^n\times \mathbb C^m
\]
\begin{Proposition}Assume that $f_1$ and $f_2$ are strongly non-degenerate mixed polynomial. Then $f$ is also strongly non-degenerate.
We assume that $f_1, f_2$ does not have any linear term so that they have a critical points at the respective origin.
Then we have
\begin{enumerate}
\item If $f_1$ and $f_2$ are locally tame  along vanishing coordinate
      subspaces, $f$ is also locally tame  along vanishing coordinate subspaces.
In particular, $f$ satisfies $a_f$ condition.
\item If $f_1$ or $f_2$ does not satisfy $a_f$-condition, $f$ does not satisfy $a_f$-condition.
\end{enumerate}
\end{Proposition}
\begin{proof}
Assume that $I_1\in I_v(f_1)$ and $I_2\in I_v(f_2)$.
Then $f|\mathbb C^{I_1}\times \mathbb C^{I_2}\equiv 0$.
Take $\De_1\in \Ga_{nc}(f_1)$ with $I(\De_1)=I_1$ and $\De_2\in \Ga_{nc}(f_2)$ with $I(\De_2)=I_2$.
Then $\De:=\De_1*\De_2\in \Ga_{nc}(f)$ and 
$f_\De(\bfz,\bar\bfz,\bfw,\bar
 \bfw)=f_{1\De_1}(\bfz,\bar\bfz)+f_{2\De_2}(\bfw,\bar\bfw)$ satisfies
 certainly the local tameness  condition. (Here $\De_1*\De_2$ is the convex polyhedron spanned by $\De_1$ and $\De_2$.
Conversely suppose that $\De\subset \Ga_{v}(f)$ with $I=I(\De)$.
Then $f|\mathbb C^I\equiv 0$. Put $I_1=I\cap\{1,\dots, n\}$ and
 $I_2=I\setminus I_1$. Then $I_1\in \mathcal I_v(f_1)$
and $I_2\in \mathcal I_v(f_2)$. 
Take  $P$ so that $I(P)=I$ and put
 $\De_1=\De\cap\mathbb C^n$ and $\De_2=\De\cap \mathbb C^m$.
Then $\De=\De_1*\De_2$.
Let $P_1,P_2$ be the projection to $\mathbb C^n$ or $\mathbb C^m$
respectively.
Then $\De(P_1)=\De_1$ and $\De(P_2)=\De_2$
and $f_P=f_{1,P_1}(\bfz,\bar\bfz)+f_{2,P_2}(\bfw,\bar\bfw)$ and it is
 certainly locally tame.
This proves (1).

To prove the assertion (2),
assume for example $f_1$ does not satisfies $a_f$-condition. Take stratification $\mathcal S$ such that its restriction to $\mathbb C^n$
and $\mathbb C^m$  are stratification $\mathcal S_1$ and $\mathcal S_2$ for $f_1$ and $f_2$ respectively. 
By the assumption, there exists $\bfp\in V_1=V(f_1)$ 
and  a stratum $M$ of $\mathcal S_1$ with $p\in M$ and 
an analytic curve $\bfz(t)$ in 
$\mathbb C^n\setminus\{\bf 0\}$ such that $\bfz(0)=\bfp$ and 
$a_f$ condition  is not satisfied along this curve.
Write $f_1=g_1+ih_1$, $f_2=g_2+i h_2$ and $f=g+ih$. We may assume that
 \[
  \bar\partial g_1(\bfz(t))=v_{g_1} ^\infty t^{s_1}+\text{(higher terms)}
\]
and  it
converges to $v_{g_1}^\infty$. We assume  
$\ord\, \bar\partial h_1(\bfz(t))\ge \ord\,\bar\partial
 g_1(\bfz(t))$. By
the same technique as in the proof of Theorem \ref{a_f-nice}, we  take a new vector 
$\bar\partial h_1':=\bar \partial h_1(t)-k(t) \bar\partial g_1(\bfz(t))$ so that
\[
 \bar\partial h_1'(t)=v_{h_1}^\infty t^{d}+\text{(higher terms)}
\] so that  the leading coefficient vectors $v_{g_1}, v_{h_1}$ 
 are linearly independent over $\mathbb R$. Thus the limit of the tangent space 
$T_{\bfz(t)}f_1\inv(f_1(\bfz(t))$ is given by 
${v_{g_1}^\infty}^{\perp}\cap {v_{h_1}^\infty}^{\perp}$.
By the assumption, we  have that 
${v_{g_1}^\infty}^{\perp}\cap {v_{h_1}^\infty}^{\perp}\not\supset T_{\bfz(0)}M$.
Note that  $d$ is   the order
of
$\bar\partial h_1'(\bfz(t))$.  Consider the analytic path
$(\bfz(t),\bfw(t))$ where $\bfw(t)=(t^{3d},\dots, t^{3d})$. Let us consider 
\[\bar\partial h_2'(\bfw(t)):=\bar\partial h_2(\bfw(t))-k(t)\bar\partial g_2(\bfw(t)).
\]
Then it is easy to see that $\ord\,\bar\partial g_2(\bfw(t)),\bar\partial h_2'(\bfw(t))\ge 2d$.
Put 
\[\bar\partial h'(\bfz(t),\bfw(t))=\partial h(\bfz(t),\bfw(t))-k(t)\bar\partial g(\bfz(t),\bfw(t)).
\]
Thus this implies that $\ord\,\bar\partial g(\bfz(t),\bfw(t))=\ord\,\bar\partial g_1(\bfz(t))$ and 
$\ord\, \bar \partial h'(\bfz(t),\bfw(t))=\ord\,\bar\partial h_1'(\bfz(t))\ge 2d$ and the normalized limits are given as
\[
\bar\partial g(\bfz(t),\bfw(t))\to (v_{g_1},{\bf{0}}),
\quad
\bar\partial h'(\bfz(t),\bfw(t))\to (v_{h_1},{\bf 0})
\]
which implies the limit of $T_{(\bfz(t),\bfw(t))}f\inv(f(\bfz(t),\bfw(t))$ is
$({v_{g_1}^\infty}^\perp\cap {v_{h_1}^\infty}^\perp)\times \mathbb C^m$.
By the assumption,
$({v_{g_1}^\infty}^\perp\cap {v_{h_1}^\infty}^\perp)\times \mathbb C^m\not\supset T_{\bfz(0)}M$.
\end{proof}
\def\cprime{$'$} \def\cprime{$'$} \def\cprime{$'$} \def\cprime{$'$}
  \def\cprime{$'$} \def\cprime{$'$} \def\cprime{$'$} \def\cprime{$'$}

\end{document}